\documentclass[preprint,authoryear,12pt]{elsarticle}

\usepackage{latexsym}
\usepackage{amsmath}
\usepackage{dsfont}
\usepackage{amsthm}
\usepackage{amsfonts}
\usepackage{amssymb}
\usepackage{psfrag}
\usepackage{graphicx}
\usepackage{verbatim}
\usepackage{color}
\usepackage{booktabs}
\usepackage{enumerate}

\usepackage{subcaption}

\usepackage[margin=1in]{geometry}
\allowdisplaybreaks

\renewenvironment{description}[1][0pt]
  {\list{}{\labelwidth=0pt \leftmargin=#1
   }}
  {\endlist}

\allowdisplaybreaks

\newtheorem{theorem}{Theorem}

\newtheorem{lemma}[theorem]{Lemma}

\newtheorem{proposition}[theorem]{Proposition}

\numberwithin{equation}{section} 
\numberwithin{theorem}{section}

\begin{document}
	\begin{frontmatter}
		\title{Upper bounds for the blow-up time of a system of fractional differential equations with Caputo derivatives and a numerical scheme for the solution of the system}
		
		\author[jv]{Jos\'{e} Villa-Morales\corref{co1}} %\fnref{fn2}}
		\ead{jvilla@correo.uaa.mx}
		\address[jv]{Departamento de Matem\'{a}ticas y F\'{\i}sica,
			Universidad Aut\'onoma de Aguascalientes,
			Av. Universidad 940, C.P. 20131 Aguascalientes, Ags., M\'exico}

     %   \fntext[fn2]{Partially supported by the grant PIM20-1 of Universidad Aut\'onoma de Aguascalientes.}
		\cortext[co1]{Corresponding author}
		
		\begin{abstract}
			The article provides upper bounds for the blow-up time of a system of fractional differential equations in the Caputo sense. Furthermore, concrete examples of blow-up time estimation are given using a numerical algorithm of the predictor-corrector type.
		\end{abstract}

		\begin{keyword}  Blow-up time, upper bounds, Caputo's fractional differential equations, predictor-corrector. \MSC 26A33, 34A08, 65R20, 34A40
		\end{keyword}
	\end{frontmatter}
	
\section{Introduction}
The study of fractional differential equations is a current topic in pure and applied mathematics, this is due in large part to the need for new techniques for their analysis and the multiple applications that this mathematical entity has (see for example \cite{KST2002}, \cite{ZLW2014}, \cite{VRR2022} and the references it contains). Quite concisely, we can say that the theoretical study of fractional differential equations has focused primarily on proving existence and uniqueness theorems (see \cite{ZC2012}, \cite{BQ2009}, \cite{SZL2012}, \cite{FT2004}, \cite{MST2007}, \cite{AS2022}) or in giving criteria for the non-existence of solutions (see \cite{KAAA2014}, \cite{LT2010}, \cite{K2010}, \cite{JMS2019}). Something similar occurs in the context of fractional partial differential equations (see \cite{KLT2005}, \cite{CV2019}). This work is related to the second topic, that is, criteria for the non-existence of global solutions.

Let us assume that $\alpha \in (0,1)$ and $q_{1}$, $q_{2}$, $p_{11}$, $p_{12}$, $p_{21}$, $p_{22}$ are non-negative real numbers. In the paper, we will consider the system of Caputo's fractional differential equations, 
\begin{eqnarray}
^{C}\hspace{-0.1cm}D_{0+}^{\alpha}x(t)&=&t^{q_{1}}x^{p_{11}}(t)y^{p_{12}}(t), \ \ t>0, \label{eq1}\\
^{C}\hspace{-0.1cm}D_{0+}^{\alpha}y(t)&=&t^{q_{2}}y^{p_{21}}(t)x^{p_{22}}(t), \ \ t>0, \label{eq2}
\end{eqnarray}
with initial conditions
\begin{equation}
x(0)>0, \ \ y(0)>0.  \label{incond}
\end{equation}
We will say that the solution $(x(t),y(t))$ of the system (\ref{eq1})-(\ref{incond}) blows-up (or explodes) in finite time if there is a positive real-number $\tau_{xy}$ such that the conditions (\ref{eq1})-(\ref{incond}) are satisfied for all $0\leq t<\tau_{xy}$ and
\begin{equation}
\lim_{t\uparrow t_{xy}} \, \sup\{|x(s)|+|y(s)|: s\in [0,t]\}=\infty.
\end{equation}
The time $\tau_{xy}$ is called blow-up (or explosion) time. The objective of this work is to give some conditions on the parameters of the system (\ref{eq1})-(\ref{eq2}) in order to obtain upper bounds of the explosion time $\tau_{xy}$, as a consequence conditions under which the system has no global solution are obtained, see Theorem \ref{ThM}.

It should be noted that, although it is a bit tedious, when $\alpha=1$ and $q_{1}=q_{2}$ the explosion time of the system of ordinary differential equations (\ref{eq1}-(\ref{eq2}) can be given explicitly. To the best of our knowledge, there are no non-trivial fractional differential equations for which their explosion time is explicitly known. In addition, it is convenient to point out that there are practical applications of modeling with fractional differential equations where the explosions occurs in a finite time. For example, this occurs in the study of the fracture of materials or in the study of the combustion of gases, see for example \cite{ZLW2014} or \cite{BE1989}; in these cases, abrupt changes occur in the phenomenon under consideration. Due to this, when modeling, it is convenient to have numerical and mathematical methods that allow us to bound the blow-up time of a system of interest.

Roughly speaking, it can be said that the difficulty of finding upper bounds for the blow-up time of the system $(x(t),y(t))$, solution of (\ref{eq1})-(\ref{eq2}), consists of comparing the growth of the components of the system, that is, verifying that, for example, after a certain time the component $x(t)$ is greater than or equal to the component $y(t)$. When this occurs, the problem is reduced, in a sense, to the study of a single fractional inequality, see Proposition \ref{propblow}. In our case, we have overcome this difficulty by introducing a new inequality related to the difference of the components of the system, so we can determine if it is positive or negative, therefore we can compare the components where the system has a solution, see the inequality (\ref{estj}). As far as we know, there are no comparison theorems for the solutions of a system of fractional differential equations that can be applied to this type of system; for this reason, the technique for studying the blow-up time, as well as the upper bounds obtained, seem new to us.

Since, in general, it is rare to obtain explicit solutions to a system of differential equations, it is necessary to consider a numerical scheme to give an idea of its solution. In our case, we adapt a numerical scheme that is a modification of the fractional Euler method that considers a corrective term. Using this numerical algorithm, it is possible to graphically estimate the blow-up time of the system (\ref{eq1})-(\ref{eq2}). Three illustrative numerical examples are discussed.

The article is organized as follows. We present some definitions and certain preliminary results in Section 2, here it is worth highlighting the important Jensen's type inequality (\ref{convexpro}). In Section 3, we study a fractional inequality and prove the main result. Finally, in Section 4, we give a numerical algorithm for solving a system of general fractional differential equations and consider three illustrative numerical examples.

\section{Preliminaries}

In what follows we will assume that $\alpha \in (0,1)$ and $a$ will be a real number. We will begin by recalling some definitions and basic properties that will be useful to us. If $u(t)$ is an absolutely continuous function, then the Caputo left hand-side fractional derivative is defined as
\begin{equation}
^{C}\hspace{-0.1cm}D_{a+}^{\alpha}u(t)=\frac{1}{\Gamma(1-\alpha)} \int_{a}^{t} \frac{u^{\prime}(s)ds}{(t-s)^{\alpha}}, \ \ t>a, \label{Capudef}
\end{equation}
where $\Gamma$ is the usual gamma function.

On the other hand, the Riemann-Liouville right hand-side fractional derivative of a function $u(t)$ is defined as
 \begin{equation*}
^{RL}\hspace{-0.1cm}D_{a-}^{\alpha}u(t)=-\frac{1}{\Gamma(1-\alpha)} \frac{d}{dt}\int_{t}^{a} \frac{u(s)ds}{(s-t)^{\alpha}}, \ \ t<a.
\end{equation*}

Another concept that we will use is the following, the Riemann-Liouville left-side fractional integral of a function $u(t)$ is defined as
\begin{equation}
I^{\alpha}_{a+}u(t)=\frac{1}{\Gamma(\alpha)} \int_{a}^{t} \frac{u(s)ds}{(t-s)^{1-\alpha}}, \ \ t>a.
\end{equation}

We compile some basic properties in the next result.

\begin{lemma}
Let $\beta$ be a real number, then
\begin{eqnarray}
\left( ^{RL}\hspace{-0.1cm}D_{a-}^{\alpha} (a-s)^{\beta-1}\right) (t)&=&
\begin{cases}
0, & \beta = \alpha,\\
\frac{\Gamma(\beta)}{\Gamma(\beta-\alpha)}(a-t)^{\beta-\alpha-1}
, & \beta \neq \alpha.
\end{cases} \label{deripot}
\end{eqnarray}

\noindent If the function $u(t)$ is continuous, then
\begin{equation}
(I^{\alpha}_{a+} \, ^{C}\hspace{-0.1cm}D_{a+}^{\alpha} u)(t)=u(t)-u(a), \ \ t>a. \label{intder}
\end{equation}
\end{lemma}

\begin{proof}
The derivative (\ref{deripot}) is formulae (2.1.19) and (2.1.20) in \cite{KST2002}. On the other hand, Lemma 2.22 of \cite{KST2002} contains formula (\ref{intder}).
\end{proof}

For $\alpha>0$ and $\beta \in \mathbb{R}$ the Mittag-Leffler function $E_{\alpha, \beta}$ is defined by
\begin{equation}
E_{\alpha,\beta}(t)= \sum_{k=0}^{\infty} \frac{t^{k}}{\Gamma(\alpha k+\beta)}, \ \ t\in \mathbb{R}. \label{MiLef}
\end{equation}
In particular, $E_{\alpha,1}$ will be denoted by $E_{\alpha}$.

We have the following elementary result.

\begin{lemma}
Let us assume that $\lambda$ and $a$ are two real numbers and set
\begin{equation}
e_{\alpha}^{\lambda (a-t)}:=(a-t)^{\alpha-1}E_{\alpha,\alpha}[\lambda (a-t)^{\alpha}], \ \ t<a. \label{defz}
\end{equation}
If $0<\alpha \leq 1$, then the function $e_{\alpha}^{\lambda (a-t)}$ is positive and 
\begin{equation}
^{RL}\hspace{-0.1cm}D_{a-}^{\alpha} e_{\alpha}^{\lambda (a-t)} - \lambda e_{\alpha}^{\lambda (a-t)}=0, \ \ t < a. \label{solpos}
\end{equation}
\end{lemma}

\begin{proof}
To verify that $e_{\alpha}^{\lambda (a-t)}$ is positive it suffices to prove that $E_{\alpha,\alpha}$ is positive. It is clear that $E_{\alpha,\alpha}(t)>0$ if $t\geq 0$. On the other hand, the function $E_{\alpha}(-t)$, with $t\geq 0$, is completely monotonic, namely
$$E_{\alpha}(-t)=\int_{0}^{\infty} e^{-ts}dF_{\alpha}(s),$$
where $F_{\alpha}(s)$ is a bounded function and $F^{\prime}_{\alpha}(s)\geq 0$ (moreover, an explicit expression for $F^{\prime}_{\alpha}(s)$ can be find in \cite{Pol}). With this we deduce
$$E_{\alpha,\alpha}(-t)=-\alpha\frac{d}{dt}E_{\alpha}(-t)=\alpha\int_{0}^{\infty} se^{-ts}dF_{\alpha}(s)> 0.$$

The equation (\ref{solpos}) is easily established using the property (\ref{deripot}) and the linearity of the fractional derivative.
\end{proof}

\begin{proposition} \label{convexpro}
Let $\gamma\geq 2$ and $u(t)$ be an absolutely continuous function on $(a,b)$. If
\begin{equation}
u(t)\geq u(0)\geq 0, \ \ t\in (a,b), \label{desjen}
\end{equation}
then
\begin{equation}
^{C}\hspace{-0.1cm}D_{a+}^{\alpha} u^{\gamma}(t) \leq \gamma u^{\gamma-1}(t)\, ^{C}\hspace{-0.1cm}D_{a+}^{\alpha} u(t), \ \ t\in (a,b). \label{covprop}
\end{equation}
\end{proposition}

\begin{proof}
The Definition $(\ref{Capudef})$ and the use of Fubini's theorem imply
\begin{eqnarray*}
u^{\gamma-1}(t)\, ^{C}\hspace{-0.1cm}D_{a+}^{\alpha} u(t) - \frac{1}{\gamma} \, ^{C}\hspace{-0.1cm}D_{a+}^{\alpha} u^{\gamma}(t)&=& \frac{1}{\Gamma(1-\alpha)} \int_{a}^{t} \frac{1}{(t-r)^{\alpha}}\left[ u^{\gamma-1}(t)u^{\prime}(r)-u^{\gamma-1}(r)u^{\prime}(r)\right]dr\\
&=& \frac{1}{\Gamma(1-\alpha)} \int_{a}^{t} \frac{u^{\prime}(r)}{(t-r)^{\alpha}}   
\int_{r}^{t} \frac{d u^{\gamma-1}}{ds}(s)dsdr\\
&=&\frac{1}{\Gamma(1-\alpha)} \int_{a}^{t} (\gamma-1)u^{\gamma-2}(s)u^{\prime}(s)  
\int_{a}^{s} \frac{u^{\prime}(r)}{(t-r)^{\alpha}}drds.
\end{eqnarray*}
Inasmuch as
$$\frac{\partial}{\partial s}\left( \int_{a}^{s}\frac{u^{\prime}(r)}{(t-r)^{\alpha}}dr\right)^{2} = \frac{2u^{\prime}(s)}{(t-s)^{\alpha}}  \int_{a}^{s}\frac{u^{\prime}(r)}{(t-r)^{\alpha}}dr,$$
then 
\begin{eqnarray*}
u^{\gamma-1}(t)\, ^{C}\hspace{-0.1cm}D_{a+}^{\alpha} u(t) - \frac{1}{\gamma} \, ^{C}\hspace{-0.1cm}D_{a+}^{\alpha} u^{\gamma}(t)&=& \frac{\gamma-1}{2 \Gamma(1-\alpha)} \int_{a}^{t}u^{\gamma-2}(s)(t-s)^{\alpha} \frac{\partial}{\partial s}\left( \int_{a}^{s}\frac{u^{\prime}(r)}{(t-r)^{\alpha}}dr\right)^{2}ds.
\end{eqnarray*}
In this way, the usual integration by parts and the inequality (\ref{desjen}) yield
\begin{eqnarray*}
u^{\gamma-1}(t)\, ^{C}\hspace{-0.1cm}D_{a+}^{\alpha} u(t) - \frac{1}{\gamma} \, ^{C}\hspace{-0.1cm}D_{a+}^{\alpha} u^{\gamma}(t)&\geq& \frac{(\gamma-1)u^{\gamma-2}(0)}{2 \Gamma(1-\alpha)} \int_{a}^{t}(t-s)^{\alpha} \frac{\partial}{\partial s}\left( \int_{a}^{s}\frac{u^{\prime}(r)}{(t-r)^{\alpha}}dr\right)^{2}ds\\
&=& \frac{\alpha(\gamma-1)u^{\gamma-2}(0)}{2 \Gamma(1-\alpha)} \int_{a}^{t}(t-s)^{\alpha-1} \left( \int_{a}^{s}\frac{u^{\prime}(r)}{(t-r)^{\alpha}}dr\right)^{2}ds.
\end{eqnarray*}
From here, the result evidently follows.
\end{proof}

The hypothesis $(\ref{desjen})$ in Proposition \ref{convexpro} is not necessary when $\gamma=2$. In this case, the following references \cite{Al2010} or \cite{ADG2014} could be consulted.

\section{Proof of the main result}

Next, our first step will be to give an estimate of the blow-up time for a fractional inequality. As is common, for $p\geq 1$ we denote by $\tilde{p}\in (1,\infty]$ the conjugate index of $p$ defined by means of
$$\frac{1}{p}+\frac{1}{\tilde{p}}=1.$$

\begin{proposition} \label{propblow}
Let $K>0$, $q\geq 0$ and $p\geq 1$. Let $u(t)$ be a non-negative solution of the following Caputo's fractional inequality
\begin{equation}
{^C}\hspace{-0.1cm}D^{\alpha}_{0+}u(t)\geq K t^{q}u^{p}(t), \ \ t>0,  \label{desgral}
\end{equation}
with initial condition
\begin{equation*}
u(0)=u_{0}> 0.
\end{equation*}
If $q+1>q \tilde{p}$, then $u(t)$ blows-up in finite time; moreover, the blow-up time $\tau_{u}$ is less than or equal to
\begin{equation}
\tau (u_{0},q,p):= \left( \frac{\Gamma(q(1-\tilde{p})+1)}{(u_{0})^{p}\Gamma(q+1)} B(\lambda_{m}) \right) ^{1/(\tilde{p}(\alpha +q))}, \label{detau}
\end{equation}
where $B(\lambda_{m})=\min \{B(\lambda): \lambda>\alpha \tilde{p}-1\}$ with 
\begin{equation}
B(\lambda):=\frac{\Gamma(\lambda+1)^{\tilde{p}-1}\,\Gamma(\lambda +1 - \alpha \tilde{p}) \, \Gamma(q+\lambda +2)}{\Gamma(\lambda +1 - \alpha)^{\tilde{p}} \,\Gamma(q+ \lambda+2-\tilde{p}(q+\alpha))}. \label{defB}
\end{equation}
\end{proposition}

\begin{proof}
We will use the well known capacity method, see for example \cite{KFT2017}. Let us assume that the solution $u(t)$ of the inequality (\ref{desgral}) is defined  in $[0,s]$, for some
\begin{equation}
s>\tau (u_{0},q,p). \label{exsol}
\end{equation}
For $\lambda >\alpha \tilde{p}-1$ let us introduce the test function
\begin{equation}
\varphi_{s}(t):=\frac{(s-t)^{\lambda}}{s^{\lambda}}, \ \ t\in (0,s).
\end{equation}
Multiplying the inequality (\ref{desgral}) by $\varphi_{s}$ and integrating from $0$ to $s$ gives
\begin{equation}
\int_{0}^{s}\varphi_{s}(t) K t^{q} u^{p}(t)dt \leq \int_{0}^{s} \varphi_{s}(t) \, ^{C}\hspace{-0.1cm}D_{0+}^{\alpha}u(t)dt.
\end{equation}
Using the integration by parts formula for the Caputo fractional derivative (see \cite{Agr2007} or \cite{KFT2017}) we have
\begin{equation}
\int_{0}^{s} \varphi_{s}(t) \, ^{C}\hspace{-0.1cm} D_{0+}^{\alpha}u(t)dt= \int_{0}^{s} u(t) \, ^{RL}\hspace{-0.1cm}D_{s-}^{\alpha}\varphi_{s}(t)dt - \frac{\Gamma(\lambda+1)}{\Gamma(\lambda-\alpha+2)} s^{1-\alpha} u(0).
\end{equation}
Using the formula
\begin{eqnarray}
\, ^{RL}\hspace{-0.1cm}D^{\alpha}_{s-} \varphi_{s}(t)
&=&\frac{\Gamma(\lambda+1)}{\Gamma(\lambda+1-\alpha)} \cdot \frac{\varphi_{s}(t)}{(s-t)^{\alpha}}, \ \ 0\leq t < s,
\end{eqnarray}
and $u(0)\geq 0$, we obtain
\begin{eqnarray*}
\int_{0}^{s}\varphi_{s}(t) K t^{q} u^{p}(t)dt &\leq& 
\frac{\Gamma(\lambda+1)K}{\Gamma(\lambda+1-\alpha)} \int_{0}^{s} u(t)  \frac{\varphi_{s}(t)}{(s-t)^{\alpha}} dt\\
&=& \frac{\Gamma(\lambda+1)K}{\Gamma(\lambda+1-\alpha)} \int_{0}^{s} \varphi_{s}(t)^{\frac{1}{p}} t^{\frac{q}{p}} u(t) \cdot \frac{\varphi_{s}(t)^{1-\frac{1}{p}}t^{-\frac{q}{p}}}{(s-t)^{\alpha}} dt.
\end{eqnarray*}
Employing the epsilon Young inequality, with $\varepsilon \in (0,\Gamma(\lambda+1 - \alpha)/\Gamma(\lambda+1)]$, see Appendix B in \cite{Eva}, we get
\begin{eqnarray*}
\int_{0}^{s}\varphi_{s}(t) K t^{q} u^{p}(t)dt &\leq&  \frac{\Gamma(\lambda+1)K}{\Gamma(\lambda+1-\alpha)} \left\lbrace 
\varepsilon \int_{0}^{s} \varphi_{s}(t) t^{q} u(t)^{p}dt + C(\varepsilon) \int_{0}^{s} \frac{\varphi_{s}(t)^{\tilde{p}\left( 1-\frac{1}{p}\right) }t^{-\frac{q \tilde{p}}{p}}}{(s-t)^{\alpha \tilde{p}}} dt\right\rbrace,
\end{eqnarray*}
where 
\begin{equation}
C(\varepsilon)=\frac{1}{(\varepsilon p)^{\tilde{p}/p}\, \tilde{p}}. \label{ceps}
 \end{equation}
An elementary algebraic manipulation leads us to the inequality
\begin{eqnarray}
\int_{0}^{s}\varphi_{s}(t) K t^{q} u^{p}(t)dt \left(1-\frac{\varepsilon\Gamma(\lambda+1)}{\Gamma(\lambda+1-\alpha)} \right) 
&\leq&  \frac{C(\varepsilon)\Gamma(\lambda+1)K}{\Gamma(\lambda+1-\alpha)} \int_{0}^{s} \frac{\varphi_{s}(t)t^{q(1-\tilde{p})}}{(s-t)^{\alpha \tilde{p}}} dt \nonumber \\
&=&\frac{C(\varepsilon)\Gamma(\lambda+1)K}{\Gamma(\lambda+1-\alpha)} \cdot B(\lambda+1-\alpha \tilde{p},q(1-\tilde{p})+1) \nonumber \\
&& \cdot \ s^{1-\alpha \tilde{p}+q(1-\tilde{p})}. \label{estIarr}
\end{eqnarray}

Otherwise, applying the operator $I^{\alpha}_{0+}$ to the inequality (\ref{desgral}) we get, by (\ref{intder}),
\begin{equation}
u(t)\geq u(0)+I^{\alpha}_{0+}(Kt^{q}u(t)^{p})\geq u_{0}, \ \ 0<t<s, \label{estpci}
\end{equation}
then
\begin{eqnarray*}
\int_{0}^{s}\varphi_{s}(t) K t^{q} u^{p}(t)dt 
&\geq& (u_{0})^{p} K \int_{0}^{s}\varphi_{s}(t) t^{q}dt \\
&=& (u_{0})^{p} K B(q+1,\lambda+1) \, s^{1+q}. 
\end{eqnarray*}
This inequality and (\ref{estIarr}) imply
\begin{eqnarray}
s^{\tilde{p}(\alpha+q)}&\leq& \frac{C(\varepsilon)\Gamma(\lambda+1)}{\Gamma(\lambda+1-\alpha)-\varepsilon \Gamma(\lambda+1)} \cdot \frac{B(\lambda+1-\alpha \tilde{p},q(1-\tilde{p})+1)}{u(0)^{p}B(q+1,\lambda+1)} \nonumber \\
&=& \frac{\Gamma(q(1-\tilde{p})+1)}{u(0)^{p}\Gamma(q+1)} \cdot \frac{\Gamma(\lambda +1 - \alpha \tilde{p}) \Gamma(q+\lambda +2)}{\Gamma(q+\lambda+2-\tilde{p}(q+\alpha))} \cdot H(\varepsilon), \label{ears}
\end{eqnarray}
where 
$$H(\varepsilon):=\frac{C(\varepsilon)}{\Gamma(\lambda +1 - \alpha)-\varepsilon \Gamma(\lambda+1)}.$$
Since the right hand side of inequality (\ref{ears}) is valid for any $0<\varepsilon < \Gamma(\lambda+1 - \alpha)/\Gamma(\lambda+1)$, then we use the identity (\ref{ceps}) to minimize the function $H(\varepsilon)$ to get
\begin{eqnarray}
s &\leq& \left( \frac{\Gamma(q(1-\tilde{p})+1)}{(u_{0})^{p}\Gamma(q+1)} \cdot  B(\lambda) \right) ^{1/(\tilde{p}(\alpha+q))}, \ \ \lambda > \alpha \tilde{p}-1, \label{estima2}
\end{eqnarray}
where $B(\lambda)$ is defined in (\ref{defB}). Taking the infimum over $\lambda$ in the above inequality we deduce that $s\leq \tau (u_{0},q,p)$, contradicting the inequality (\ref{exsol}). In this way, the desired result is achieved.
\end{proof}

We are now in a position to state and prove our main result.

\begin{theorem} \label{ThM}
Let $\alpha \in (0,1)$ and $q_{1}$, $q_{2}$, $p_{11}$, $p_{12}$, $p_{21}$, $p_{22}$ be non-negative real numbers. We will consider the system of Caputo's fractional differential equations,
\begin{eqnarray}
^{C}\hspace{-0.1cm}D_{0+}^{\alpha}x(t)&= &t^{q_{1}}x^{p_{11}}(t)y^{p_{12}}(t), \ \ t>0, \label{nsis1}\\
^{C}\hspace{-0.1cm}D_{0+}^{\alpha}y(t)&=&t^{q_{2}}y^{p_{21}}(t)x^{p_{22}}(t), \ \ t>0,\label{nsis2} 
\end{eqnarray}
with initial conditions
\begin{equation*}
x(0)=x_{0}>0, \ \ y(0)=y_{0}>0.
\end{equation*}
\begin{description}
\item[Case $q_{1} \neq q_{2}$:] Let us set $q_{i}:=\min\{q_{1},q_{2}\}$, $j:=3-i$ and
$$p_{j}:=p_{ji}+p_{jj}\gamma_{j},$$
where
$$\gamma_{j}:=\frac{p_{ij}+1-p_{ji}}{2}.$$
If 
\begin{equation}
p_{ij}\geq 3 + p_{ji}, \ \  p_{ii}+1\geq p_{jj}, \ \ q_{j}+1>q_{j} \tilde{p}_{j}, \label{casog}
\end{equation}
then the solution of system (\ref{nsis1})-(\ref{nsis2}) blows-up in finite time. Moreover, the blow-up time, $\tau_{xy}$, is less than or equal to $\tau (u_{j},q_{j},p_{j})$, where $\tau (u_{0},p,q)$ is defined in (\ref{detau}) and 
\begin{eqnarray*}
u_{j}&=&
\begin{cases}
x_{0}, & j = 1,\\
y_{0}, & j=2.
\end{cases}
\end{eqnarray*}
\item[Case $q_{1} = q_{2}$:] In this case, if 
\begin{equation}
p_{22}\geq 3 + p_{11}, \ \ p_{21}+1\geq p_{12}, \ \ q_{1}+1>q_{1} \tilde{p}_{1}, \label{caso1}
\end{equation}
then the solution of system (\ref{nsis1})-(\ref{nsis2}) blows-up in finite time and the blow-up time $\tau_{xy}$ is less than or equal to $\tau(x_{0},q_{1},p_{1})$, where 
$$p_{1}:=p_{11}+p_{12}\gamma_{1}, \ \ \gamma_{1}:=\frac{p_{22}+1-p_{11}}{2}.$$
Or, if 
\begin{equation}
p_{12}\geq 3 + p_{21}, \ \ p_{11}+1\geq p_{22}, \ \ q_{2}+1>q_{2} \tilde{p}_{2}, \label{caso2}
\end{equation}
then the solution of system (\ref{nsis1})-(\ref{nsis2}) blows-up in finite time and the blow-up time $\tau_{xy}$ is less than or equal to $\tau(y_{0},q_{2},p_{2})$, where
$$ p_{2}:=p_{21}+p_{22}\gamma_{2}, \ \ \gamma_{2}:=\frac{p_{12}+1-p_{21}}{2}.$$
\end{description}

\end{theorem}

\begin{proof}
The local existence of a positive solution $(x(t),y(t))$ of system (\ref{nsis1})-(\ref{nsis2}) is obtained by proceeding as in the classical case (i.e., the non-fractional one). Indeed, instead of treating the system of fractional differential equations, the associated system of integral equations is studied (see formula (3.5.4) in \citep{KST2002}). A first step consists of applying a version of Banach's contraction principle to a mapping, determined by the system of integral equations, in order to find a positive solution $(x(t),y(t))$ in a certain interval $[0,\delta)$, that is, we obtain a local solution (see Theorem 3.25 in \citep{KST2002}, for example). For the second step, let $[0,t_{\max})$ be the maximum interval of existence of the solution $(x(t),y(t))$. If $t_{\max}< \infty$, then $\lim_{t\uparrow t_{\max}} \, \sup\{x(s)+y(s):s\in[0,t]\} =\infty$. If such a limit were finite, then we can proceed as in the first step and deduce that we can extend the solution to the time interval $[0,t_{\max}+\delta)$, which is absurd (see, for example, the proof of Theorem $1.4$ in Chapter $6$ of \cite{Pazy}). Therefore, the maximum time of existence of the solution is precisely the blow-up (explosion) time.

Now, we will concentrate on finding upper bounds for the blow-up time $\tau_{xy}$ of the system (\ref{nsis1})-(\ref{nsis2}). Without loss of generality, we will suppose that we have $q_{2}\geq q_{1}$, $p_{12}\geq 3 + p_{21}$, $p_{11}+1\geq p_{22}$ and $q_{2}+1>q_{2} \tilde{p}_{2}$; in the other cases the procedure is similar, therefore we omit them. We will see that the blow-up time $\tau_{xy}$ of the system is less than or equal to $\tau (y_{0},q_{2},p_{2})$. We proceed by contradiction, that is, $\tau(y_{0},q_{2},p_{2})<\tau_{xy}$, then the system  (\ref{nsis1})-(\ref{nsis2}) has a solution $(x(s),y(s))$ in $[0,t]$, for some 
\begin{equation}
\tau(y_{0},q_{2},p_{2})<t<\tau_{xy}. \ \ \label{conpblow}
\end{equation}
Under this assumption it makes sense to introduce the following function
$$J(s)=Mx(s)-y^{\gamma_{2}}(s), \ \ 0 \leq s \leq t,$$
where
$M$ is a positive constant that will be fixed letter. Observe that the hypotheses imply $\gamma_{2}\geq 2$ and this is a condition necessary to apply Proposition \ref{convexpro}. Using the linearity of the Caputo's fractional derivative, (\ref{covprop}), (\ref{nsis1}) and (\ref{nsis2}) we obtain
\begin{eqnarray*}
^{C}\hspace{-0.1cm}D_{0+}^{\alpha}J (s) &=& M\,  ^{C}\hspace{-0.1cm}D_{0+}^{\alpha} x(s) - \, ^{C}\hspace{-0.1cm}D_{0+}^{\alpha} y^{\gamma_{2}}(s)\\
 &\geq &   M \, ^{C}\hspace{-0.1cm}D_{0+}^{\alpha}  x(s) - \gamma_{2} \, y^{\gamma_{2}-1}(s)\, ^{C}\hspace{-0.1cm}D_{0+}^{\alpha} y(s)\\
 &=& M s^{q_{1}}x^{p_{11}}(s)y^{p_{12}}(s) - \gamma_{2} \, s^{q_{2}}x^{p_{22}}(s)y^{\gamma_{2}-1+p_{21}}(s).
\end{eqnarray*}
Setting 
$$h(s):=s^{q_{2}} x^{p_{11}}(s) y^{\gamma_{2}-1+p_{21}}(s), \ \ 0\leq s \leq t,$$
the above inequality implies
\begin{eqnarray*}
^{C}\hspace{-0.1cm}D_{0+}^{\alpha}J (s) + h(s) J(s)&\geq & x^{p_{11}}(s)y^{p_{12}}(s)(M s^{q_{1}}-s^{q_{2}})\\
&& + \ s^{q_{2}}y^{\gamma_{2}-1+p_{21}}(s)(M x^{p_{11}+1}(s)-\gamma_{2} \, x^{p_{22}}(s)).
\end{eqnarray*}
Proceeding as in the proof of (\ref{estpci}) we can verify that
$$x(s)\geq x_{0}>0, \ \ 0 \leq s \leq t.$$
If we take
\begin{eqnarray*}
M &>& \max\left\lbrace s^{q_{2}-q_{1}}, \gamma_{2} \, x^{p_{22}-p_{11}-1}(s): 0\leq s \leq t \right\rbrace\\
&=&\max\left\lbrace t^{q_{2}-q_{1}}, \gamma_{2} \, (x_{0})^{p_{22}-p_{11}-1} \right\rbrace,
\end{eqnarray*}
then
\begin{eqnarray*}
^{C}\hspace{-0.1cm}D_{0+}^{\alpha}J (s) +h(s)J(s)>0, \ \ 0\leq s \leq t.
\end{eqnarray*}
Furthermore, if we take $M > (y_{0})^{\gamma_{2}} (x_{0})^{-1}$, then $J(0)>0$.

Let us set
\begin{equation*}
A:=\{r\in [0,t] : J(s)> 0, \ \text{for all} \ s\in [0,r]\}.
\end{equation*}
Notice that $0\in A$. Let us suppose that $a:=\sup A <t$. This implies
\begin{eqnarray}
^{C}\hspace{-0.1cm}D_{0+}^{\alpha}J (s) +||h|| J(s) > 0, \ \ 0\leq s < a, \label{estj}
\end{eqnarray}
where $||h||:=\sup\{|h(s)|:0\leq s \leq t\}$.
We will consider the function $e_{\alpha}^{-||h|| (a-s)}>0$ defined in (\ref{defz}). Multiplying by $e_{\alpha}^{-||h|| (a-s)}$ on both sides of the inequality (\ref{estj}) and integrating with respect to $s$ we obtain 
\begin{eqnarray*}
0&<&\int_{0}^{a} e_{\alpha}^{-||h|| (a-s)}\, \left[ ^{C}\hspace{-0.1cm}D_{0+}^{\alpha} J(s)+||h|| J(s)\right] ds\\
&=&\int_{0}^{a} e_{\alpha}^{-||h|| (a-s)}\, ^{C}\hspace{-0.1cm}D_{0+}^{\alpha} J(s)ds + \int_{0}^{a} e_{\alpha}^{-||h|| (a-s)} ||h|| J(s)ds\\
&\leq& \int_{0}^{a} J(s)\, ^{RL}\hspace{-0.1cm}D_{a-}^{\alpha} e_{\alpha}^{-||h|| (a-s)} \,ds + \int_{0}^{a} J(s) ||h|| e_{\alpha}^{-||h|| (a-s)} \,ds\\
&=& \int_{0}^{a} J(s)\, \left[ ^{RL}\hspace{-0.1cm}D_{a-}^{\alpha} e_{\alpha}^{-||h|| (a-s)} + ||h||\, e_{\alpha}^{-||h|| (a-s)} \right]  ds=0.
\end{eqnarray*}
This contradiction yields $a=t$, then $J(s)\geq 0$, $0\leq s \leq t$, namely
\begin{equation*}
x(s)\geq \frac{1}{M}\, y^{\gamma_{2}}(s), \ \ 0 \leq s \leq t.
\end{equation*}
From the above inequality and (\ref{nsis2}) we arrive at
\begin{equation*}
^{C}\hspace{-0.1cm}D_{0+}^{\alpha} y(s)\geq \frac{1}{M^{p_{22}}}\, s^{q_{2}}y^{p_{21}+p_{22}\gamma_{2}}(s) , \ \ 0 \leq s \leq t. 
\end{equation*}
If we take $K=M^{-p_{22}}$, $q=q_{2}$, $p=p_{21}+p_{22}\gamma_{2}$ and $u_{0}=y_{0}$ in Proposition \ref{propblow} we conclude that the system (\ref{nsis1})-(\ref{nsis2}) can not have a solution on $[0,t]$, otherwise we have the inequality, $t\leq \tau(u_{0},q,p)=\tau(y_{0},q_{2},p_{2})$. But this contradicts (\ref{conpblow}), as expected.
\end{proof}

\section{Numerical experiments}

To find the solution of the system (\ref{nsis1})-(\ref{nsis2}), in this section we are going to consider a numerical scheme. We will use the predictor-corrector approach introduced in \cite{DFF2002} to find numerical solutions. In fact, in such paper it is also included a code that is very simple to implement. It is convenient to point out that in \cite{DFF2002} the algorithm is introduced to find the solution of a single fractional differential equation, in our case, we adapt their algorithm for systems of fractional equations. In addition, we present the algorithm for a more general system than the one considered in Section $3$, since we believe that it may be of interest to a broader audience. 

We are going to consider a numerical scheme for the system of fractional differential equations in Caputo's sense,
\begin{eqnarray}
^{C}\hspace{-0.1cm}D_{0+}^{\alpha}x(t)&=&f(t,x(t),y(t)), \ \ t>0, \label{geq1}\\
^{C}\hspace{-0.1cm}D_{0+}^{\alpha}y(t)&=&g(t,x(t),y(t)), \ \ t>0, \label{geq2}
\end{eqnarray}
where $f(t,x,y)$ and $g(t,x,y)$ are given functions together with the initial data
\begin{equation}
x(0)=x_{0}, \ \ y(0)=y_{0}.  \label{gincond}
\end{equation}

Let us look at a numerical solution in the time interval $[0,T]$. We denote by $N$ the number of divisions of $[0,T]$ and let $t_{n}:=nh$, $n=0,1,...,N$, be the corresponding regular partition, where $h:=T/N$. By $(x_{n+1},y_{n+1})$, we set the numeric approximation of the solution $(x(t),y(t))$ of the system (\ref{geq1})-(\ref{geq2}) at time $t=t_{n+1}$. For $j=0,1,...,n+1$, we will set 
\begin{align*}
x_{n+1}=x_{0}+\frac{h^{\alpha}}{\Gamma(\alpha+2)}\sum_{j=0}^{n}a_{j,n+1}f(t_{j},x_{j},y_{j})+\frac{h^{\alpha}}{\Gamma(\alpha+2)}f(t_{n+1},p_{n+1},q_{n+1}),\\
y_{n+1}=y_{0}+\frac{h^{\alpha}}{\Gamma(\alpha+2)}\sum_{j=0}^{n}a_{j,n+1}g(t_{j},x_{j},y_{j})+\frac{h^{\alpha}}{\Gamma(\alpha+2)}g(t_{n+1},p_{n+1},q_{n+1}),
\end{align*}
where 
\begin{eqnarray*}
a_{j,n+1}&=&
\begin{cases}
n^{\alpha+1}-(n-\alpha)(n+1)^{\alpha}, & j = 0,\\
(n-j+2)^{\alpha+1}+(n-j)^{\alpha+1}-2(n-j+1)^{\alpha+1}, & 1\leq j \leq n,
\end{cases}
\end{eqnarray*}
and 
\begin{gather*}
p_{n+1}=x_{0}+\frac{1}{\Gamma(\alpha)}\sum_{j=0}^{n}b_{j,n+1}f(t_{j},x_{j},y_{j}),\\
q_{n+1}=y_{0}+\frac{1}{\Gamma(\alpha)}\sum_{j=0}^{n}b_{j,n+1}g(t_{j},x_{j},y_{j}),\\
\end{gather*}
here
\begin{equation*}
b_{j,n+1}=\frac{h^{\alpha}}{\alpha}\left( (n+1-j)^{\alpha}-(n-j)^{\alpha}\right).
\end{equation*}

The blow-up time that will be obtained from the numerical scheme is denoted by $t_{\text{num}}$. More precisely, such time is deduced from the graph of the numerical solution of the system of interest. It should be noted that the objective of this work is not to design a numerical scheme to determine the blow-up time, this is a wide topic of study, see for example \cite{PV2022} and the references there in. Our objective here is to show that the explosion time obtained in Theorem \ref{ThM}, which we will denote by $\tau_{ub}$, is an upper bound for the numerical blow-up time, $t_{\text{num}}$.

We are going to consider three examples, which correspond to the three possible cases that can occur in Theorem \ref{ThM}.

\begin{description}[0cm]
\item[Example 1.] Let us examine the system of fractional differential equations
\begin{eqnarray}
^{C}\hspace{-0.1cm}D_{0+}^{\alpha}x(t)&= &t^{0.5}x^{1.5}(t)y^{3.6}(t), \ \ t>0, \label{sisex11} \\
^{C}\hspace{-0.1cm}D_{0+}^{\alpha}y(t)&=&t^{1.5}y^{0.5}(t)x^{2.4}(t), \ \ t>0, \label{sise12}
\end{eqnarray}
with initial conditions
\begin{equation*}
x(0)=1, \ \ y(0)=1.2.
\end{equation*}

We apply the predictor-corrector algorithm to system (\ref{sisex11})-(\ref{sise12}) to obtain the graph of its solution, which can be seen in Figure \ref{GraE1.1}. From the graphical representation we obtain an estimate of the numerical explosion time $t_{\text{num}}$, for the fractional index $\alpha = 0.1, 0.4, 0.6, 0.9$.
\begin{figure}[!ht]    
    \begin{subfigure}[t]{.5\textwidth}
    \includegraphics[width=3in,height=2.25in]{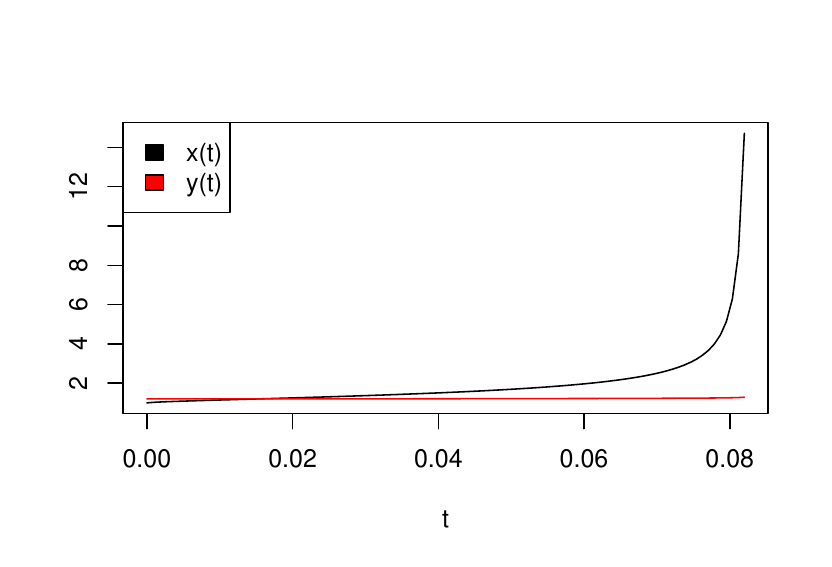}
    \caption{$\alpha=0.1$, \ $t_{\text{num}}=0.85$.}  
    \end{subfigure}
    \hspace{.5cm}
     \begin{subfigure}[t]{.5\textwidth}
    \includegraphics[width=3in,height=2.25in]{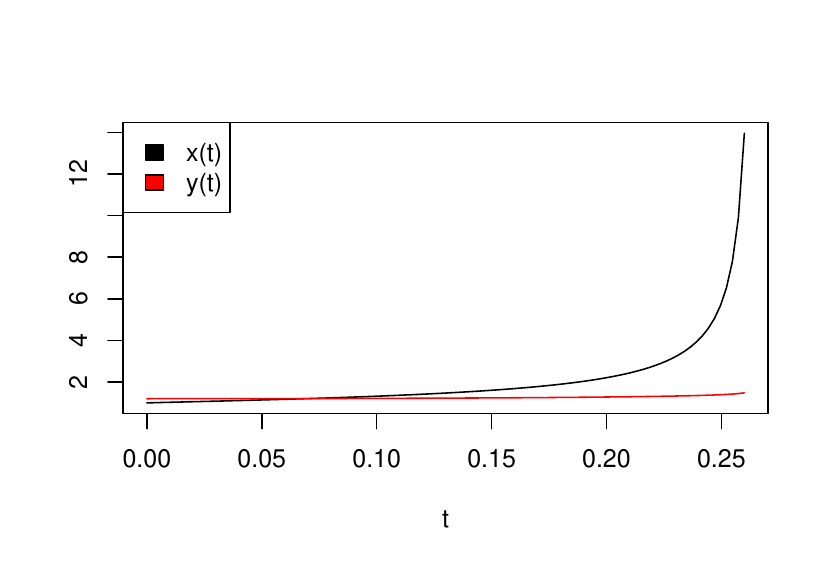}
    \caption{$\alpha=0.4$, \ $t_{\text{num}}=0.28$.}  
    \end{subfigure}
    \begin{subfigure}[t]{.5\textwidth}
    \includegraphics[width=3in,height=2.25in]{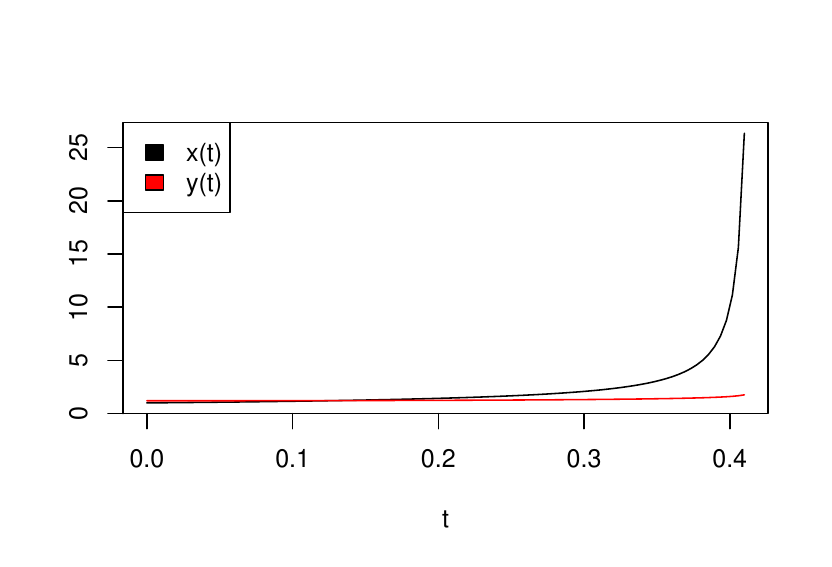}
    \caption{$\alpha=0.6$, \ $t_{\text{num}}=0.44$.}  
    \end{subfigure}
    \hspace{.5cm}
     \begin{subfigure}[t]{.5\textwidth}
    \includegraphics[width=3in,height=2.25in]{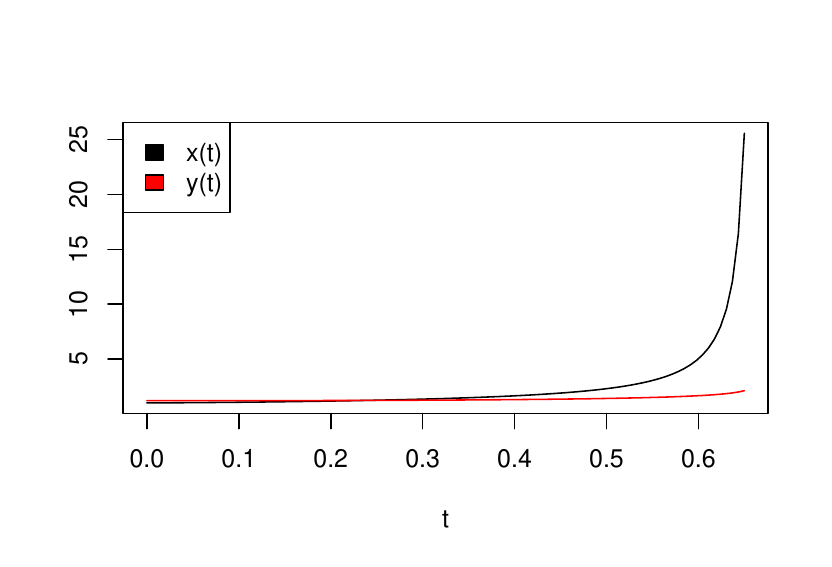}
    \caption{$\alpha=0.9$, \ $t_{\text{num}}=0.67$.}  
    \end{subfigure}
 \vspace{0.4cm}
   \caption{Graphs of the solutions of system (\ref{sisex11})-(\ref{sise12}).}   \label{GraE1.1} 
\end{figure}

On the other hand, the parameters 
$$q_{1}=0.5, \ \  p_{11}=1.5, \ \ p_{12}=3.6, \ \ q_{2}=1.5, \ \ p_{21}=0.5, \ \ p_{22}=2.4,$$
satisfy the conditions of (\ref{casog}), so we can apply Theorem \ref{ThM}.  In order to do this, we need to know the minimum value of the function $B$, defined in (\ref{defB}). Because such a function is quite complicated, it is not obvious that a minimum value exists. Fortunately, when graphing the function $B$, see Figure \ref{GraE1.2}, we note that this function is convex, for each value of $\alpha$, therefore the existence of the value sought follows. The function $B$ has a similar behaviour in the following two examples, so we will omit its graph. With the minimum value of $B$, we get the upper bound  $t_{ub}$ of the blow-up time $\tau_{xy}$.
\begin{figure}[!ht]
    \begin{subfigure}[t]{.5\textwidth}
    \includegraphics[width=3in,height=2.25in]{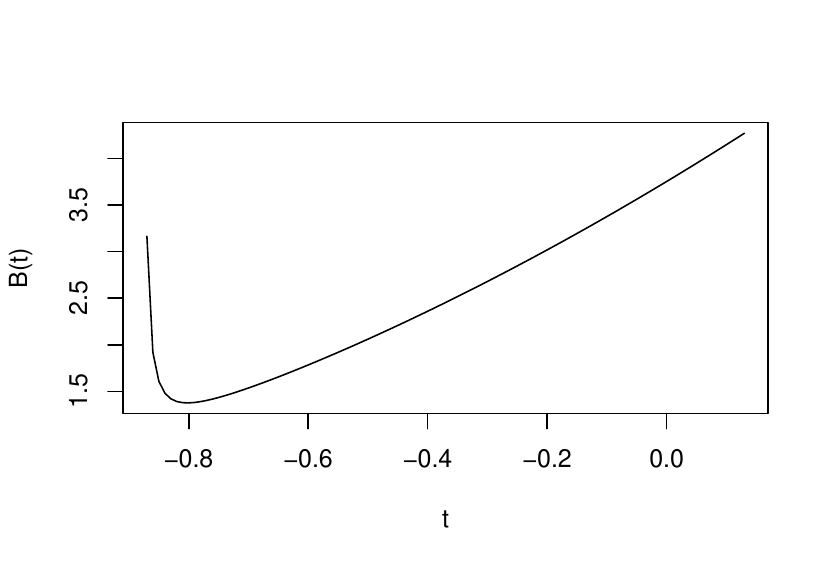}
    \caption{ $\alpha=0.1$, \ $\lambda_{m}=-0.802...$}  
    \end{subfigure}
    \hspace{.5cm}
     \begin{subfigure}[t]{.5\textwidth}
    \includegraphics[width=3in,height=2.25in]{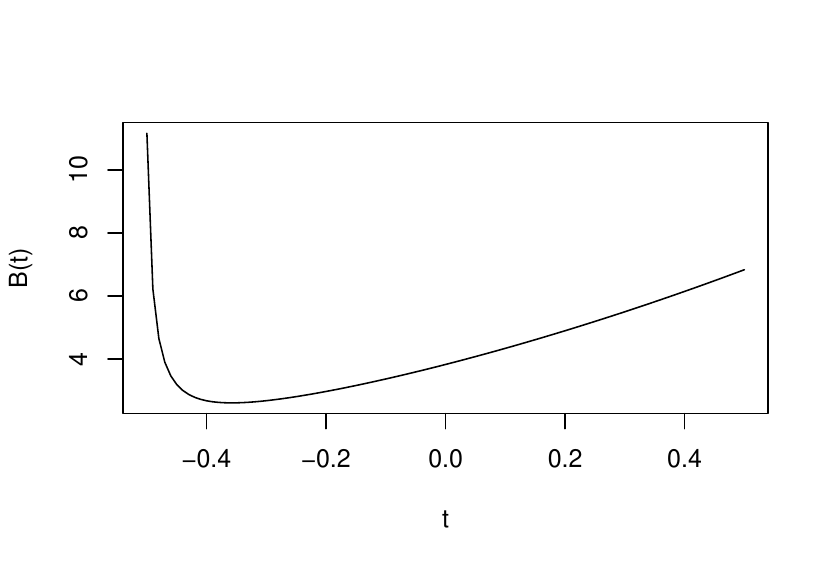}
    \caption{$\alpha=0.4$, \ $\lambda_{m}=-0.358...$}  
    \end{subfigure}
    \begin{subfigure}[t]{.5\textwidth}
    \includegraphics[width=3in,height=2.25in]{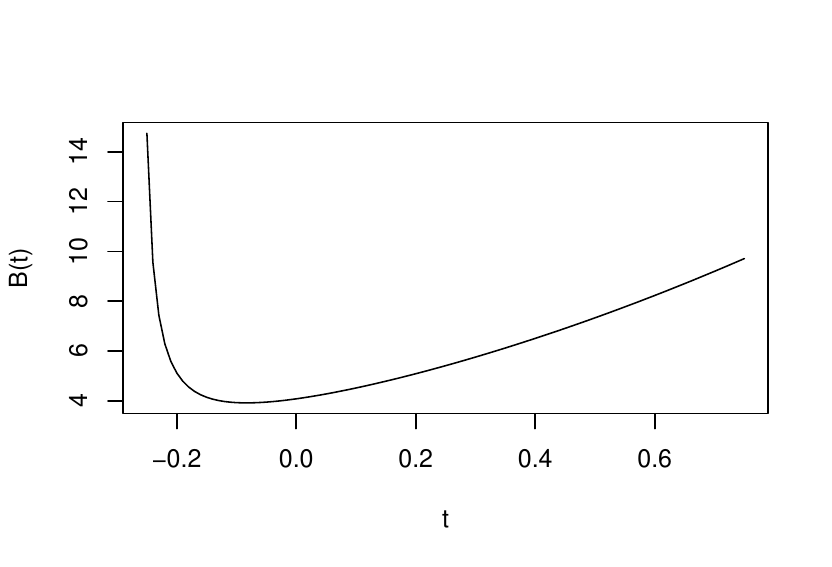}
    \caption{$\alpha=0.6$, \ $\lambda_{m}=-0.083...$}  
    \end{subfigure}
    \hspace{.5cm}
     \begin{subfigure}[t]{.5\textwidth}
    \includegraphics[width=3in,height=2.25in]{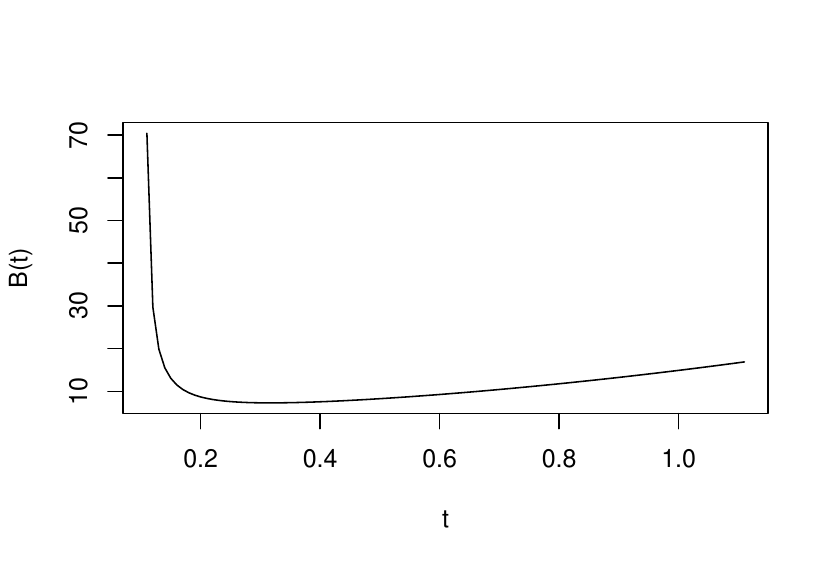}
    \caption{$\alpha=0.9$, \ $\lambda_{m}=0.315...$}  
    \end{subfigure}
 \vspace{0.4cm}
   \caption{Graphs of $B$ given in Example $1$.}   \label{GraE1.2} 
\end{figure}

We concentrate this information in the following table:
\begin{center}
\begin{tabular}{ |c||c|c|c|c| } 
 \hline
 $\alpha$ & $0.1$ & $0.4$ & $0.6$ & $0.9$ \\ 
 \hline
 $\lambda_{m}$  & $-0.802...$ & $-0.358...$ & $-0.083...$ & $0.315...$ \\ 
 \hline
  $t_{\text{num}}$  & $0.085$ & $0.28$ & $0.44$ & $0.66$ \\ 
 \hline
 $\tau_{ub}$  & $0.720...$ & $0.998...$ & $1.169...$ & $1.415...$ \\ 
 \hline
\end{tabular}
\end{center}
From this, we clearly appreciate that upper bounds have been obtained for the blow-up (explosion) time. In graph (d) of Figure 1 it can be seen that the function $y(t)$ begins to grow; in fact, if it is graphed separately, it can be seen that it also explodes in finite time. Its blow-up time is approximately $0.69$.

\item[Example 2.] Now let us consider the system of fractional differential equations
\begin{eqnarray}
^{C}\hspace{-0.1cm}D_{0+}^{\alpha}x(t)&= &y^{3.2}(t), \ \ t>0, \label{sisex21} \\
^{C}\hspace{-0.1cm}D_{0+}^{\alpha}y(t)&=&y^{0.2}(t)x^{0.5}(t), \ \ t>0, \label{sisex22}
\end{eqnarray}
with initial conditions
\begin{equation*}
x(0)=0.5, \ \ y(0)=0.5.
\end{equation*}
Applying the predictor-corrector algorithm, the graphs of the solutions of system (\ref{sisex21})-(\ref{sisex22}), for the parameter $\alpha = 0.1, 0.4, 0.6, 0.9$, are obtained, and they appear in Figure \ref{GraE2}. From here we estimate the numerical value of the explosion time, $t_{\text{num}}$.
\begin{figure}[!ht]   
    \begin{subfigure}[t]{.5\textwidth}
    \includegraphics[width=3in,height=2.25in]{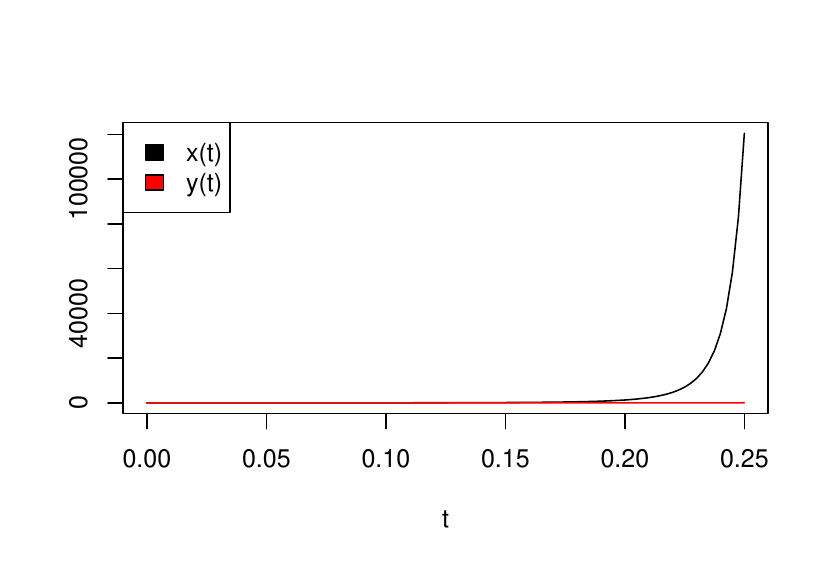}
    \caption{$\alpha=0.1$, \ $t_{\text{num}}=0.35$.}  
    \end{subfigure}
    \hspace{.5cm}
     \begin{subfigure}[t]{.5\textwidth}
    \includegraphics[width=3in,height=2.25in]{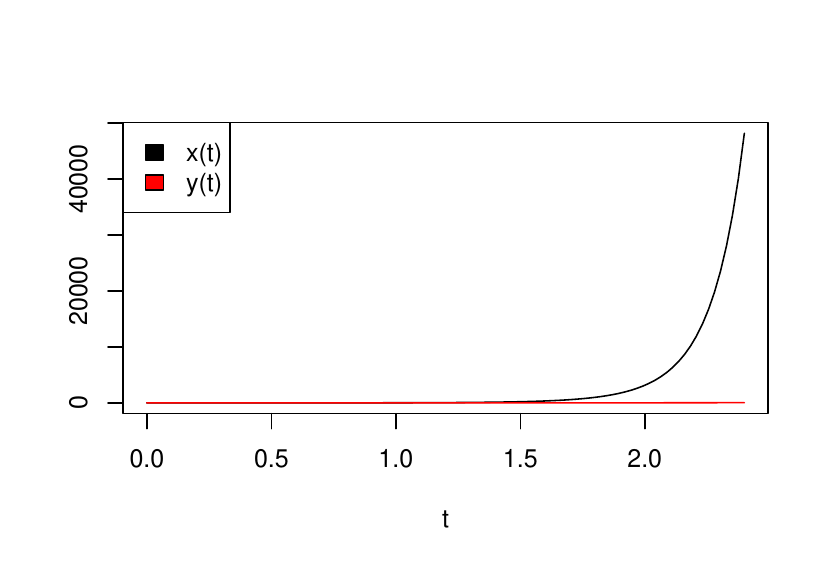}
    \caption{$\alpha=0.4$, \  $t_{\text{num}}=3.8$.}  
    \end{subfigure}
    \begin{subfigure}[t]{.5\textwidth}
    \includegraphics[width=3in,height=2.25in]{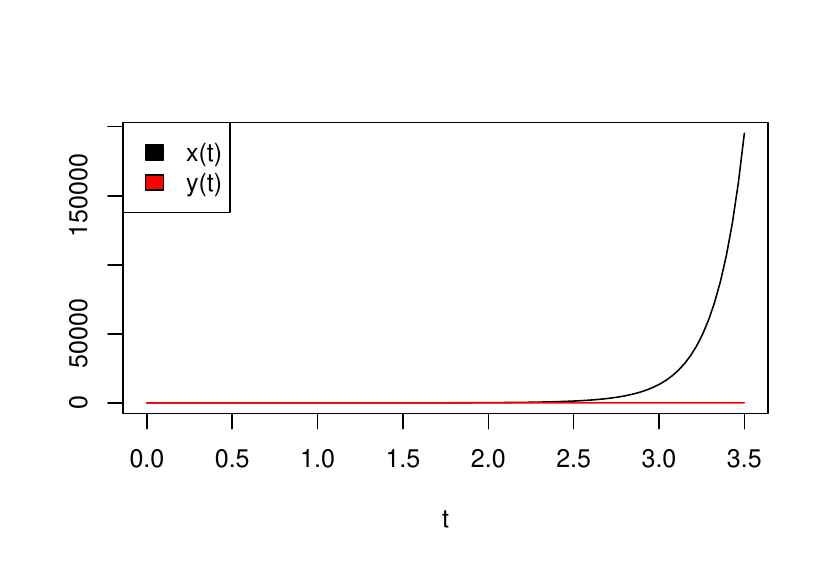}
    \caption{$\alpha=0.6$, \ $t_{\text{num}}=5.1$.}  
    \end{subfigure}
    \hspace{.5cm}
     \begin{subfigure}[t]{.5\textwidth}
    \includegraphics[width=3in,height=2.25in]{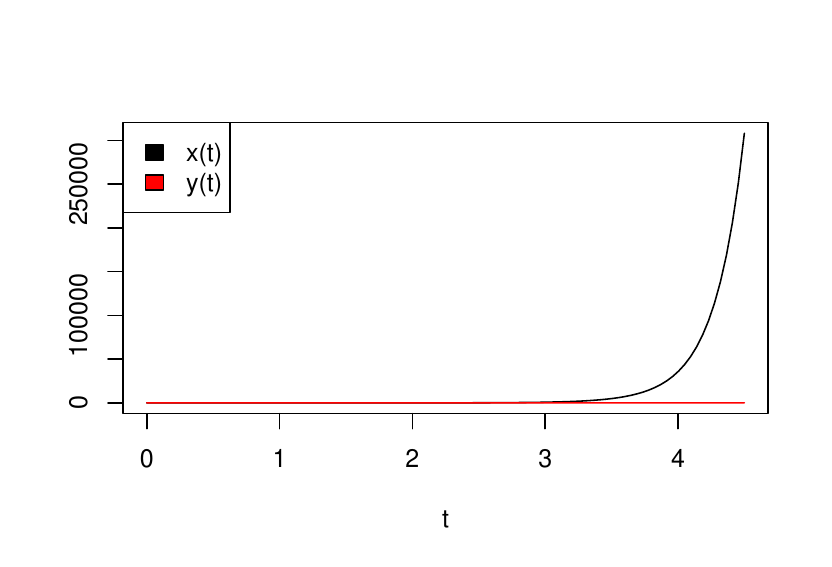}
    \caption{$\alpha=0.9$, \ $t_{\text{num}}=6.9$.}  
    \end{subfigure}
 \vspace{0.4cm}
   \caption{Graphs of the solutions of system (\ref{sisex21})-(\ref{sisex22}).}   \label{GraE2} 
\end{figure}

We observe that the parameters
$$q_{1}=0, \ \ p_{11}=0, \ \ p_{12}=3.2, \ \ q_{2}=0, p_{21}=0.2, \ \ p_{22}=0.5,$$
of the system of equations  (\ref{sisex21})-(\ref{sisex22}) meet, in this case, conditions (\ref{caso1}), therefore we can apply Theorem \ref{ThM} to obtain an estimate $\tau_{ub}$ from above the blow-up time $\tau_{xy}$ of this system. We summarize the results in the following table:
\begin{center}
\begin{tabular}{ |c||c|c|c|c| } 
 \hline
 $\alpha$ & $0.1$ & $0.4$ & $0.6$ & $0.9$ \\ 
 \hline
   $t_{\text{num}}$  & $0.35$ & $3.8$ & $5.1$ & $6.9$ \\ 
 \hline
 $\tau_{ub}$  & $8.899...$ & $6.333...$ & $7.297...$ & $8.948...$ \\ 
 \hline
\end{tabular}
\end{center}

As a result, we can conclude that $\tau_{ub}$ is an upper bound for the numerical blow-up time, $t_{\text{num}}$. As in the previous case, if we graph the function $y(t)$ separately we can determine its explosion time, which in this case is approximately $8.7$. Remember that the explosion time of a system is the minimum of the explosion times of each component.

\item[Example 3.] Here we deal with the system of fractional differential equations
\begin{eqnarray}
^{C}\hspace{-0.1cm}D_{0+}^{\alpha}x(t)&= &t^{0.5}x(t)y^{3}(t), \ \ t>0, \label{sisex31}\\
^{C}\hspace{-0.1cm}D_{0+}^{\alpha}y(t)&=&t^{0.5}y^{2}(t)x^{4}(t), \ \ t>0, \label{sisex32}
\end{eqnarray}
with initial conditions
\begin{equation*}
x(0)=1, \ \ y(0)=1.
\end{equation*}
As in the previous cases, we apply the predictor-corrector algorithm to obtain the graphical solutions of system (\ref{sisex31})-(\ref{sisex32}); we present them in Figure \ref{GraE3}. Numerical estimates for time $t_{\text{num}}$ are obtained from them.
\begin{figure}[!ht]   
    \begin{subfigure}[t]{.5\textwidth}
    \includegraphics[width=3in,height=2.25in]{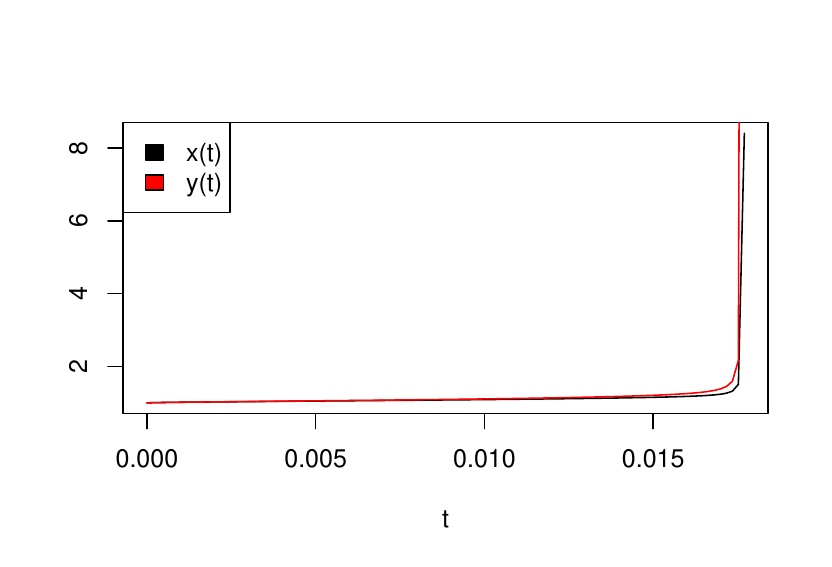}
    \caption{$\alpha=0.1$, \ $t_{\text{num}}=0.019$.}  
    \end{subfigure}
    \hspace{.5cm}
     \begin{subfigure}[t]{.5\textwidth}
    \includegraphics[width=3in,height=2.25in]{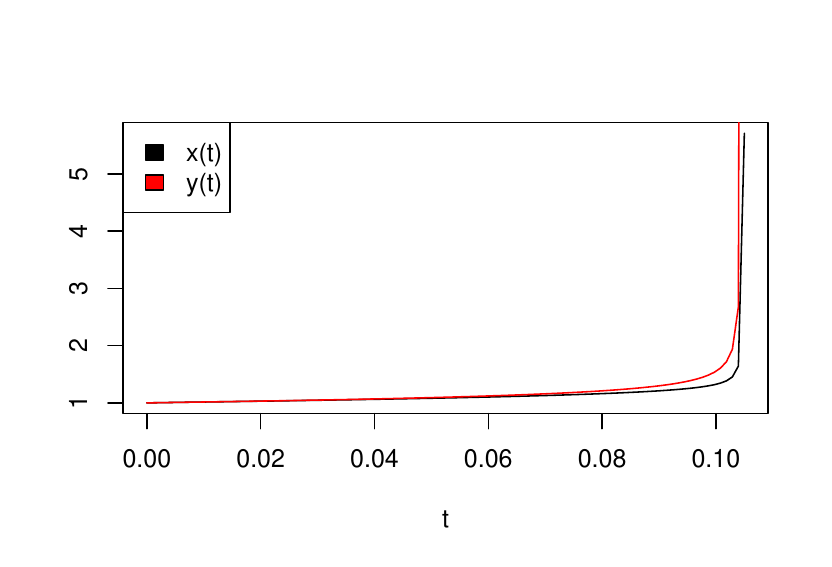}
    \caption{$\alpha=0.4$, \ $t_{\text{num}}=0.11$.}  
    \end{subfigure}
    \begin{subfigure}[t]{.5\textwidth}
    \includegraphics[width=3in,height=2.25in]{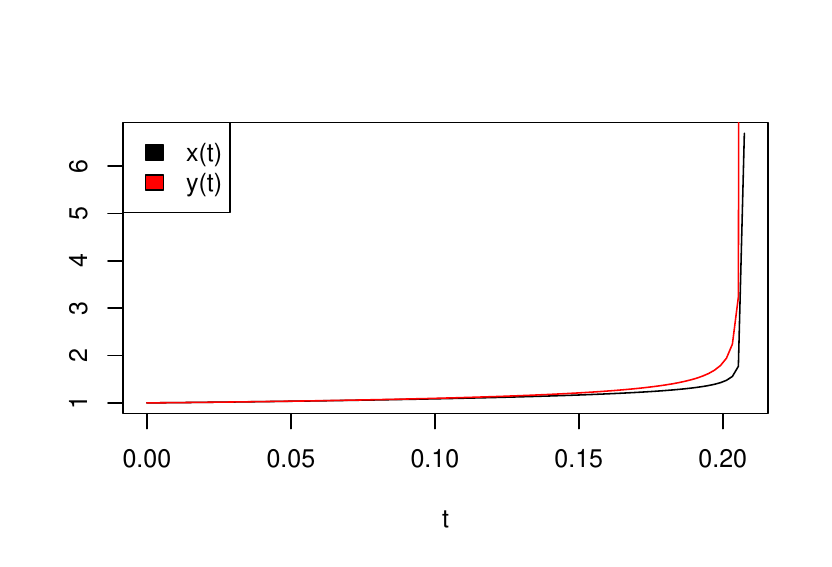}
    \caption{$\alpha=0.6$, \  $t_{\text{num}}=0.21$.}  
    \end{subfigure}
    \hspace{.5cm}
     \begin{subfigure}[t]{.5\textwidth}
    \includegraphics[width=3in,height=2.25in]{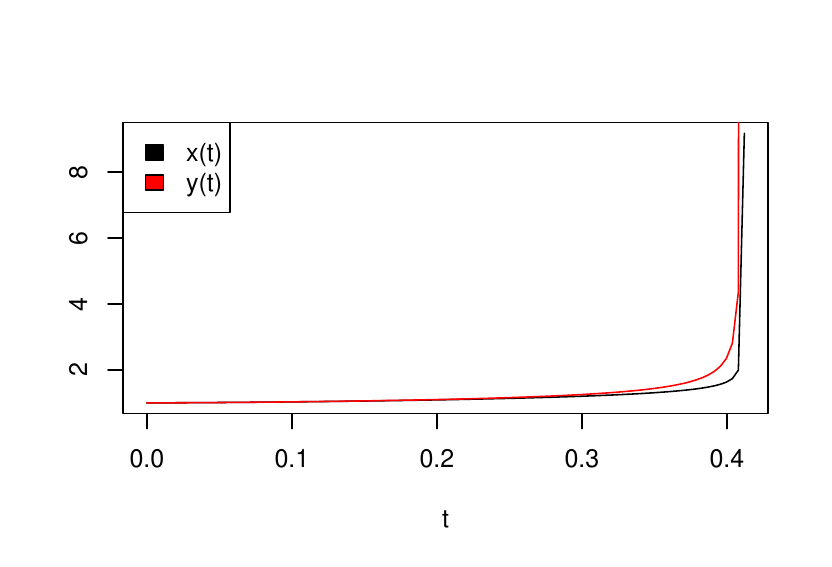}
    \caption{$\alpha=0.9$, \ $t_{\text{num}}=0.42$.}  
    \end{subfigure}
 \vspace{0.4cm}
   \caption{Graphs of the solutions of system (\ref{sisex31})-(\ref{sisex32}).}   \label{GraE3} 
\end{figure}
The parameters of the system (\ref{sisex31})-(\ref{sisex32}) are
$$q_{1}=0.5, \ \ p_{11}=0
1, \ \ p_{12}=3, \ \ q_{2}=0.5, p_{21}=2, \ \ p_{22}=4,$$
and they meet the conditions imposed in (\ref{caso2}), therefore we can apply Theorem \ref{ThM} to obtain $\tau_{ub}$, which provides us with an upper bound for the numerical time explosion, $t_{\text{num}}$. We put this information together in the following table:
\begin{center}
\begin{tabular}{ |c||c|c|c|c| } 
 \hline
 $\alpha$ & $0.1$ & $0.4$ & $0.6$ & $0.9$ \\ 
 \hline
 $t_{\text{num}}$  & $0.019$ & $0.11$ & $0.21$ & $0.42$ \\ 
 \hline
$\tau_{ub}$  & $1.228...$ & $1.551...$ & $1.726...$ & $1.967...$ \\ 
 \hline
\end{tabular}
\end{center}
Here we again see that $\tau_{ub}$ is an upper bound for $t_{\text{num}}$. In this case, both components have approximately the same explosion time. 

\end{description}

\section{Conclusions}

In the present work, upper bounds have been obtained for the blow-up time of a system of fractional differential equations in the Caputo sense (\ref{eq1})-(\ref{eq2}). As a consequence of this, sufficient conditions have been derived so the system of fractional differential equations does not have a global solution. Furthermore, based on the predictor-corrector algorithm, a numerical scheme has been proposed to solve systems of fractional differential equations in general. Using this numerical method, three illustrative examples have been presented that correspond to each of the cases of the Theorem \ref{ThM}. 

\bigskip
\bigskip
\noindent \textit{Acknowledgment:} 
The author was partially supported by the grant PIM22-1 of Universidad Aut\'{o}noma de Aguascalientes. Thanks to Jos\'e Miguel Villa-Ocampo for his help in the implementation of the numerical algorithm in R.

\bibliographystyle{elsarticle-harv}
\bibliography{FraSis2}

\end{document}